\documentclass[a4paper,12pt]{article}
\usepackage[margin=1in]{geometry}
\usepackage{amsmath,amsthm,amssymb,amsfonts}
\usepackage[hidelinks]{hyperref}

\newtheorem{proposition}{Proposition}
\newtheorem{theorem}{Theorem}

\title{An optimal refinement-compatible bijection between singleton-free partitions and partitions without cyclic adjacencies}
\author{Vuong Bui\thanks{\texttt{bui.vuong@yandex.ru}}}
\date{}

\begin{document}

\maketitle

\begin{abstract}
    It is well known that the number of partitions of $[n]$ without singletons equals the number of partitions of $[n]$ in which no block contains two cyclically adjacent elements $i,i+1\pmod{n}$. Bernhart remarked that there might be no simple bijection between these two classes. Although Callan later constructed an algorithmic bijection proving the stronger equidistribution of singletons and adjacencies, his construction proceeds through multiple rounds of exchanges. Therefore, Bernhart's remark may still retain some validity, as suggested by Chen and Wang. In this article, we address this remark by giving a direct ``one-round'' bijection between the two classes. Unlike Callan's bijection, our map is closely compatible with the refinement order on partitions: in one direction it only decomposes blocks, while its inverse only merges blocks, with a single exceptional pair when $n>2$ is even. We further observe that this exception is unavoidable, establishing the optimality of the bijection with respect to the refinement order. The specific local form of these operations --- splitting off only singleton blocks and merging a singleton only with the block containing its cyclic neighbor --- also ensures that the construction restricts, without modification, to a bijection between the corresponding classes of noncrossing partitions.
\end{abstract}

\section{Introduction}

The number of partitions of $[n]=\{1,\dots,n\}$ without singletons is well known to be equal to the number of partitions of $[n]$ without blocks containing two cyclically adjacent elements $i,i+1 \pmod{n}$, using elementary approaches like the inclusion--exclusion principle or recurrences \cite{bernhart1999catalan}. In fact, the two sequences share the same entry OEIS A000296 \cite{OEISA000296}.
We call these two classes singleton-free partitions and adjacency-free partitions in this article. The latter class has several names in the literature. These
partitions are also counted by the graphical Bell number of the cycle \cite{duncan2009bell}. In fact, the case of other graphs could also be interesting. Bernhart \cite{bernhart1999catalan} addressed this coincidence and remarked that there might be no simple bijection between these two sets. Callan \cite{callan2005conjugates} gave a bijection that even shows that the number of partitions of $[n]$ with $k$ singletons is the same as the number of partitions of $[n]$ with $k$ adjacencies. The bijection, which is algorithmic in nature, uses multiple rounds in which singletons and adjacencies are interchanged in each round. Further study with formulae involving this parameter $k$ is given in \cite{mansour2014set}. Although Callan's algorithm is combinatorial, it is somewhat involved and it may not contradict Bernhart's remark. Indeed, if we use such an algorithm to prove the original coincidence of two numbers, we still have to run multiple rounds to turn the singletons in an adjacency-free partition into adjacencies, and vice versa.
Chen and Wang \cite{chen2011singletons} remarked that there is still some truth in Bernhart's remark, despite Callan's beautiful result.
The purpose of this article is to provide a simple ``one-round'' bijection that addresses the requirement of simplicity in Bernhart's remark. The bijection, which may distort the numbers of singletons and adjacencies greatly, does not prove as strong a result as Callan's algorithm. However, the ``one-round'' nature can be seen to be more elementary and more transparent. An advantage over Callan's bijection is that our bijection only decomposes blocks into smaller blocks in one direction and merges blocks in the other direction, with only one exceptional pair.
This gives a nice property on the refinement order of set partitions. Recall that the set of partitions of $[n]$, ordered by refinement, forms a
partition lattice, see \cite[Chapter 3]{stanley2011enumerative}. We write $P_1 \preceq P_2$ if every block of $P_1$ is contained in a block of $P_2$.
In other words, we have the following result, which even works for noncrossing partitions.
\begin{theorem}
\label{thm:main}
    There exists an explicit bijection between singleton-free partitions $P$ of $[n]$ and adjacency-free partitions $Q$ of $[n]$ so that $Q\preceq P$ for every matching pair $(P,Q)$, with precisely one exceptional pair when $n>2$ is even. Moreover, the exception is unavoidable for such an $n$ in any bijection. The bijection also restricts naturally to noncrossing partitions without modification.
\end{theorem}

The fact that the bijection also restricts to noncrossing partitions suggests that it satisfies an additional structural property beyond its compatibility with
the refinement order. In particular, in the process of decomposition, we will replace a block $B$ by a subset of $B$ and several singletons, so no new crossing can be created. In the other direction of merging blocks, we only merge a singleton $\{k\}$ to the block containing $k+1 \pmod{n}$. Such an operation also preserves noncrossingness: since $k,k+1\pmod{n}$ are consecutive, any crossing involving the newly added element $k$ would already yield a crossing involving $k+1$.

Our bijection, together with the proof of Theorem \ref{thm:main}, is presented in Section \ref{sec:bijection}. 

We can observe for now that one exceptional pair is unavoidable.
Indeed, suppose we have a perfect block-refining bijection (with no exception) for an even $n>2$. Consider the two singleton-free partitions (which are also non-crossing)
\begin{equation*}
    \begin{gathered}
    \left\{\{1,2\}, \{3,4\},\dots, \{n-1,n\}\right\},\\
    \left\{\{n,1\},\{2,3\},\dots, \{n-2,n-1\}\right\}.
    \end{gathered}
\end{equation*}
For each of the partitions, the perfect block-refining bijection must decompose every block into smaller blocks to eliminate adjacencies. We obtain the same result in either case, that is, the partition containing only singletons, contradicting the one-to-one correspondence.

As we have shown that no bijection can avoid an exceptional pair for such an $n$, we can say that our bijection is optimal with respect to the refinement order.
Note that Callan's bijection does not have this property, where elements are removed and merged between different blocks.

\section{A bijection between singleton-free partitions and adjacency-free partitions}
\label{sec:bijection}
This section proves Theorem \ref{thm:main}. 
The general idea is roughly: For every maximal sequence of consecutive elements in a block of a singleton-free partition, we remove every second element and create a singleton for it. This eliminates all adjacencies while still retaining the necessary information for recovery. The other direction naturally follows by looking at sequences of consecutive singletons and appropriately joining them.

We consider $n\ge 3$ to avoid some degenerate cases. (The cases $n=1,2$ can be checked directly.)
We also exclude some other boundary cases that will not be convenient to deal with later. As we will see in the main bijection, these exclusions avoid the cases in which the singleton-free partition consists of the single block $[n]$ and the adjacency-free partition contains only singletons. In particular, for even $n$, the excluded singleton-free partitions are 
\begin{equation*}
    \begin{gathered}
    \{[n]\},\\
    \left\{\{1,2\}, \{3,4\},\dots, \{n-1,n\}\right\},\\
    \left\{\{n,1\},\{2,3\},\dots, \{n-2,n-1\}\right\},
    \end{gathered}
\end{equation*}
and the excluded adjacency-free partitions are
\begin{equation*}
    \begin{gathered}
    \left\{\{1\},\{3\}, \{5\},\dots, \{n-1\}, \{2,4,6,\dots,n\}\right\},\\
    \left\{\{1\},\{2\},\dots,\{n\}\right\},\\
    \left\{\{2\},\{4\},\{6\},\dots, \{n\}, \{1,3,5,\dots,n-1\}\right\}.
    \end{gathered}
\end{equation*}
For odd $n$, the only excluded singleton-free partition is $\{[n]\}$ and the excluded adjacency-free partition is $\left\{\{1\},\{2\},\dots,\{n\}\right\}$. (Note that all these special partitions are noncrossing.) We can just match them one by one in the above order in the final bijection. One can notice that we have inserted the partition
\[
    \left\{\{1\},\{2\},\dots,\{n\}\right\}
\]
in the middle to make sure that we will have the only exception mentioned in Theorem \ref{thm:main}: the third pair does not satisfy the refinement order.

Now comes the main bijection when the boundary cases are not considered:
\begin{itemize}
    \item a map $f$: a singleton-free partition $P$ $\longrightarrow$ an adjacency-free partition $Q$.

    In each block of $P$, we call a sequence $i,\dots,i+t \pmod{n}$ of length at least $2$ an adjacency sequence if we cannot extend it to a longer sequence. As we have excluded the case of the block being the whole $[n]$, the choice of $i$ is unique. For such a sequence, we remove elements $i+t-1,i+t-3,\dots, i+t-(2k+1) \pmod{n}$ where $k$ is the largest integer so that the elements still belong to the adjacency sequence. In other words, we remove the next-to-last element and delete every second element backward along the cyclic order as many as possible. For each of the removed elements, we create a singleton containing that element. Applying this to all adjacency sequences in every block (in any order), we make the partition adjacency-free. The result is independent of the order, since distinct adjacency sequences are disjoint, and the operation on one such sequence does not affect any other.

    \begin{quote}
        For example, if $n=10$ and a block contains an adjacency sequence $8,9,10,1,2$, we remove $9,1$ from the block and add two new blocks $\{9\}$ and $\{1\}$. Another example is the sequence $10,1$: we simply remove $10$ and add a new block $\{10\}$, which makes the original block become a singleton $\{1\}$.
    \end{quote}

    The map $f$ is an injection as we can recover $P$ from $Q$ as in the map $g$ in the other direction, which will be presented later. The condition to proceed with the recovery is that $Q$ does not contain only singletons. Indeed, a block in the original partition with at least $3$ elements will have at least $2$ elements after the operations. Therefore, if the resulting partition contains only singletons, the original partition contains only blocks of size $2$ and each of them is decomposed during the operations. In such a situation, $n$ must be even and the only possibilities for the original partitions are
    \[
        \left\{\{1,2\}, \{3,4\},\dots, \{n-1,n\}\right\}
    \]
    and
    \[
        \left\{\{n,1\},\{2,3\},\dots, \{n-2,n-1\}\right\},
    \]
    which we have already excluded. We will verify that the recovered partition is precisely the original partition later. For now we can note that $f(P)\preceq P$ since we only decompose blocks.
    
    \item a map $g$: an adjacency-free partition $Q$ $\longrightarrow$ a singleton-free partition $P$

    Inside $Q$, we call a sequence $\{i\},\dots,\{i+t\} \pmod{n}$ a singleton sequence if we cannot extend it to a longer sequence.
    As we have excluded the case of the partition containing only singletons, the choice of $i$ is unique.
    If the length of such a sequence is even, we join them pair by pair, that is, $\{i,i+1\}, \{i+2,i+3\},\dots, \{i+t-1,i+t\}$. If the length is odd, we join them pair by pair until the next-to-last element, that is, $\{i,i+1\}, \{i+2,i+3\},\dots, \{i+t-2,i+t-1\}$, and we join the last element $i+t$ to the block containing $i+t+1$. Applying this to every singleton sequence (in any order), we make the partition singleton-free. The result is independent of the order in which the singleton sequences are processed. Indeed, the operations either join disjoint pairs of singleton blocks or add the final elements of odd-length singleton sequences to nonsingleton blocks, and additions to the same target block commute. 

    \begin{quote}
    For example, if the partition of $[8]$ is
    \[
        \left\{\{5\},\{6\},\{7\},\{8\},\{1\}, \{2,4\}, \{3\}\right\},
    \]
    where there are two singleton sequences, the longer one being
    \[
        \{5\},\{6\},\{7\},\{8\},\{1\}
    \]
    and the shorter one being $\{3\}$. The map gives
    \[
        \left\{\{5,6\},\{7,8\},\{1,2,3,4\}\right\},
    \]
    where the last block is obtained by joining $1$ to $\{2,4\}$ first when considering the longer singleton sequence, and joining $3$ to it later, when considering the shorter singleton sequence.
    \end{quote}

    The map is an injection as we can recover $Q$ from $P$ as in the map $f$ in the other direction. The condition to proceed with the recovery is that $P$ does not consist of the single block $[n]$. In order to obtain the block $[n]$, we need to join singletons into another block $B$. In such a situation, the singletons themselves cannot form any singleton sequence of length greater than $1$, since otherwise, the operation would create a block of size $2$. The block $B$ must be adjacency-free, and it is also the complement of the union of all the singletons. In other words, $n$ must be even and the only possibilities for the original partitions are
    \[
        \left\{\{1\},\{3\}, \{5\},\dots, \{n-1\}, \{2,4,6,\dots,n\}\right\}
    \]
    and
    \[
        \left\{\{2\},\{4\},\{6\},\dots, \{n\}, \{1,3,5,\dots,n-1\}\right\},
    \]
    which we have excluded. We will verify that the recovered partition is precisely the original partition later. For now we can note that $Q\preceq g(Q)$ since we only merge blocks.
\end{itemize}

To finish the proof, we now verify that the two maps are inverse.
\begin{proposition}
    The maps $f,g$ are inverse.
\end{proposition}
\begin{proof}
Indeed,
\begin{itemize}
    \item $g(f(P))=P$.

    Given the resulting $Q=f(P)$, we show that the process recovers $P$. If a singleton sequence of $Q$ contains only one singleton $\{j\}$, its adjacent elements $j-1,j+1\pmod{n}$ stay inside some non-singleton blocks (not necessarily together). We join $j$ to the block of $j+1$ to recover the original position of $j$.
    If we have a singleton sequence $\{j\},\{j+1\}, \{j+2\}, \dots, \{j+t\} \pmod{n}$ of length at least $2$, the element $j-1$ stays inside some non-singleton block. This excludes the possibility that we have the singleton $\{j\}$ due to removing $j-1$ from the block $\{j-1,j\}$ in $P$. Therefore, the only possibility is that the original block in $P$ is $\{j,j+1\}$, since $\{j+1\}$ is also a singleton in $Q$ while every block of size at least $3$ retains at least two elements. We join $j$ and $j+1$ into a block and continue with the remaining $\{j+2\},\dots,\{j+t\}\pmod{n}$ recursively. Since every singleton is joined back to the original position, we have proved $g(f(P))=P$.

    \item $f(g(Q))=Q$.

    Given the resulting $P=g(Q)$, we show that the process recovers $Q$. If a singleton sequence in $Q$ has even length, the corresponding blocks of the pairs in $P$ will be decomposed accordingly. In other words, every singleton in an even-length sequence is recovered. A similar situation holds for an odd-length sequence, except the last singleton $\{k\}$, which is joined to the block $B$ of $Q$ containing $k+1$. 
    Let $r$ be the largest nonnegative integer such that, for every $0\le t\le r$, we have $k+1+2t\in B$ and $\{k+2t\}$ is a singleton. (Note that $r$ is well defined, since if these conditions held for every $t$, we would obtain one of the excluded partitions listed at the beginning when $n$ is even, whereas it is impossible for odd $n$.)
    It follows that the adjacency sequence containing $k$ after the joining operation contains $k,k+1,\dots,k+2r,k+2r+1$ but does not contain $k+2r+2$. It follows that the element $k$ will be removed to create a singleton in the process of decomposing $P$. Since every singleton is recovered, we have proved $f(g(Q))=Q$.\qedhere
\end{itemize} 
\end{proof}

\bibliographystyle{unsrt}
\bibliography{singcons}
\end{document}